\documentclass{amsart}

\usepackage{amsmath,amssymb,amsthm,amscd}
\usepackage{bm}
\usepackage{color}
\usepackage{mathrsfs}
\usepackage{upgreek}
\usepackage[all]{xy}

\newtheorem{theorem}{Theorem}

\newtheorem{prop}[theorem]{Proposition}
\newtheorem{lemma}[theorem]{Lemma}
\newtheorem{cor}[theorem]{Corollary}

\newtheorem{rem}[theorem]{Remark}

\newtheorem{ques}[theorem]{Question}

\numberwithin{theorem}{section}
\numberwithin{equation}{section}
\newcommand{\W}{\mathcal{W}}
\newcommand{\Z}{\mathbb{Z}}
\newcommand{\C}{\mathbb{C}}
\newcommand{\Hc}{\mathcal{H}}
\newcommand{\V}{\mathcal{V}}
\newcommand{\E}{\mathcal{E}}
\newcommand{\F}{\mathcal{F}}
\newcommand{\g}{\mathfrak{g}}
\newcommand{\h}{\mathfrak{h}}
\newcommand{\gl}{\mathfrak{gl}}
\newcommand{\slf}{\mathfrak{sl}}
\newcommand{\e}{\operatorname{e}}
\newcommand{\ad}{\operatorname{ad}}
\newcommand{\str}{\operatorname{str}}
\newcommand{\Ker}{\operatorname{Ker}}
\newcommand{\Vtau}{V^{\tau_{k}}(\g_0)}
\newcommand{\Fne}{\Phi(\g_{\frac{1}{2}})}
\newcommand{\prin}{\mathrm{prin}}
\newcommand{\scr}{{\bf Q}}
\newcommand{\der}{\partial}
\newcommand{\Com}{\operatorname{Com}}
\newcommand{\bk}{{\bf k}}

\title{Ito's conjecture and the coset construction for $\W^k(\slf(3|2))$}

\author{Naoki Genra}
\address{Department of Mathematical and Statistical Sciences,University of Alberta, Edmonton, AB T6G 2G1, Canada.}
\email{genra@ualberta.ca}
\thanks{N.G. was supported by Grant-in-Aid for JSPS Fellows (No.17J07495) and the Research Institute for Mathematical Sciences, a Joint Usage/Research Center located in Kyoto University, and is supported by JSPS Overseas Research Fellowships.}

\author{Andrew R. Linshaw}
\address{Department of Mathematics, University of Denver, Denver, CO 80208}
\email{andrew.linshaw@du.edu}
\thanks{A. L. is supported by Simons Foundation Grant \#318755.}

\begin{document}

{\abstract \noindent Many $\W$-(super)algebras which are defined by the generalized Drinfeld-Sokolov reduction are also known or expected to have coset realizations. For example, it was conjectured by Ito that the principal $\W$-superalgebra $\W^k(\slf(n+1|n))$ is isomorphic to the coset of $V^{l+1}(\gl_n)$ inside $V^{l}(\slf_{n+1}) \otimes \E(n)$ for generic values of $l$. Here $\E(n)$ denotes the rank $n$ $bc$-system, which carries an action of $V^1(\gl_n)$, and $k$ and $l$ are related by $(k + 1) (l + n + 1) = 1$. This conjecture is known in the case $n=1$, which is somewhat degenerate, and we shall prove it in the first nontrivial case $n=2$. As a consequence, we show that the simple quotient $\W_k(\slf(3|2))$ is lisse and rational for all positive integers $l>1$. These are new examples of rational $\W$-superalgebras.}

\maketitle

\section{Introduction}
Let $\g$ be a basic classical Lie superalgebra over $\mathbb{C}$, $f$ a nilpotent element with even parity, $k$ a complex number and
\begin{align*}
\Gamma:\g=\bigoplus_{j\in\frac{1}{2}\Z}\g_j
\end{align*}
a good grading for $f$. Then the (affine) $\W$-algebras $\W^k(\g,f;\Gamma)$ are defined as $\frac{1}{2}\Z_{\geq 0}$-graded vertex superalgebras by generalized Drinfeld-Sokolov reductions associated with $\g, f,\Gamma, k$ \cite{FF,KRW}. By \cite{KW3}, we have the Miura map
\begin{align*}
\mu\colon\W^k(\g,f;\Gamma)\rightarrow\Vtau\otimes\Fne,
\end{align*}
where $\Vtau$ is the affine vertex superalgebra of $\g_0$ with its invariant bilinear form $\tau_k$ (see (2.2) in \cite{G} for the definition of $\tau_k$), and $\Fne$ is the neutral vertex superalgebra associated with $\g_{\frac{1}{2}}$ (see \cite{KRW} for the definition of $\Fne$). If $k$ is generic, $\W^k(\g,f;\Gamma)$ is simple, so $\mu$ is injective and the image of $\mu$ may be described as the intersection of kernels of screening operators; see \cite{G}. By \cite{KW3, KW4}, if $\{ u_i\}_{i=1}^{\dim\g^f}$ is a basis of $\g^f$ with $u_i\in\g_{j_i}$, where $\g^f$ is the centralizer of $f$ in $\g$, the $\W$-algebra $\W^k(\g,f;\Gamma)$ has a minimal set $\{ \widetilde{J}_{u_i}\}_{i=1}^{\dim\g^f}$ of strong generators, and then $\widetilde{J}_{u_i}$ has the conformal degree $1-j_i$ and the same parity as $u_i$, i.e. $\W^k(\g,f;\Gamma)$ is of type $\W(1-j_1,\ldots,1-j_{\dim\g^f})$. See e.g. \cite[Section 2]{CL} for definitions of strong generators and type of vertex superalgebras.

It is often useful to give an alternative realization of $\W^k(\g,f;\Gamma)$ using the {\it coset construction}. Recall that given a vertex algebra $\mathcal{V}$ and a vertex subalgebra $\mathcal{A}\subset \mathcal{V}$, the coset of $\mathcal{A}$ in $\mathcal{V}$ is defined by $$\Com(\mathcal{A}, \mathcal{V}) := \{v\in \mathcal{V}|\ [a(z), v(w)] = 0,\  \forall a \in \mathcal{A}\}.$$ This was introduced by Frenkel and Zhu in \cite{FZ}, and it generalizes earlier constructions in representation theory \cite{KP} and physics \cite{GKO}. Suppose that $L$ and $L'$ are the Virasoro elements of $\mathcal{V}$ and $\mathcal{A}$, respectively. Then under mild hypotheses, $L - L'$ is the Virasoro element of $\Com(\mathcal{A}, \mathcal{V})$, and $v \in \mathcal{V}$ lies in $\Com(\mathcal{A}, \mathcal{V})$ if and only if $L'_{-1} v = 0$; see Theorems 5.1 and 5.2 of \cite{FZ}. These results also hold if $\mathcal{V}$ and $\mathcal{A}$ are vertex superalgebras, and if the conformal degree grading is by $\frac{1}{2} \mathbb{N}$ rather than $\mathbb{N}$. It is expected that $\Com(\mathcal{A}, \mathcal{V})$ will inherit properties of $\mathcal{A}$ and $\mathcal{V}$ such as lisseness and rationality, although general results of this kind are not yet known. We say the vertex subalgebras $\mathcal{A}, \mathcal{B}$ of $\mathcal{V}$ form a {\it dual pair} if $\mathcal{B} = \Com(\mathcal{A}, \mathcal{V})$ and $\mathcal{A} = \Com(\mathcal{B}, \mathcal{V})$; in this case, $\mathcal{A} \otimes \mathcal{B}$ is conformally embedded in $\mathcal{V}$.

Given a simple Lie algebra $\g$, let $\W^k(\g)$ denote the universal $\W$-algebra of $\g$ associated to this principal nilpotent element and principal gradation. We denote by $\W_k(\g)$ its unique simple graded quotient. It was proven by Arakawa \cite{Ar2,Ar3} that $\W_k(\g)$ is lisse and rational when $k$ is a non-degenerate admissible level. These $\W$-algebras are called the {\it minimal series principal $\W$-algebras} since in the case $\g=\slf_2$ they are exactly the minimal series Virasoro vertex algebras. They are not necessarily unitary, but if $\g$ is simply laced, there exists a sub-series called the {\it discrete series} which were conjectured for many years to be unitary. 

For a simple Lie algebra $\g$ and $l\in\C$, let $V^l(\g)$ be the universal affine vertex algebra associated with $\g$ of level $l$, and let $L_l(\g)$ be its unique simple quotient which is graded by conformal degree. We denote by $u(z)$ the generating fields of $V^l(\g)$ for $u\in\g$. By abuse of notation, we also denote by $u(z)$ the generating fields of $L_l(\g)$. There exists a diagonal homomorphism $$V^{l+1}(\g) \rightarrow V^l(\g)\otimes L_1(\g),\qquad u(z) \mapsto u(z) \otimes 1 + 1 \otimes u(z),\qquad u \in \g,$$  and if we identify $V^{l+1}(\g)$ with its image inside $V^l(\g)\otimes L_1(\g)$, $\text{Com}(V^{l+1}(\g),V^l(\g)\otimes L_1(\g))$ is well defined. The following result was proven in \cite{ACL}.

\begin{theorem} \label{ACLmain}
Let $\g$ be simply laced and let $k,l$ be complex numbers related by \begin{equation*} k+h^{\vee}=\frac{l+h^{\vee}}{l+h^{\vee}+1},\end{equation*} where $h^{\vee}$ is the dual Coxeter number of $\g$. 
\begin{enumerate}
\item For generic values of $k$, we have a vertex algebra isomorphism 
$$\W^{k}(\g)\cong  \text{Com}(V^{l+1}(\g),V^l(\g)\otimes L_1(\g)),$$ and  $V^{l+1}(\g)$ and  $\W^{k}(\g)$ form a dual pair inside $ V^l(\g)\otimes L_1(\g)$.
\item Suppose that $l$ is an admissible level for $\hat{\g}$. Then $k$ defined above is a non-degenerate admissible level for $\hat{\g}$ so that $\W_{k}(\g)$ is a minimal series $\W$-algebra. We have a vertex algebra isomorphism $$\W_{k}(\g)\cong  \text{Com}(L_{l+1}(\g),L_l(\g)\otimes L_1(\g)),$$ and  $L_{l+1}(\g)$ and  $\W_{k}(\g)$ form a dual pair  in $L_l(\g)\otimes L_1(\g)$. 
\end{enumerate}
\end{theorem}

This was conjectured in \cite{BBSS} in the case of discrete series, which correspond to the case $k \in \mathbb{N}$, and by Kac and Wakimoto \cite{KW1,KW2} for arbitrary minimal series $\W$-algebras. The conjectural character formula of \cite{FKW} for minimal series representations of $\W$-algebras that was proved in \cite{Ar1}, together with the character formula of \cite{KW2} of branching rules, proves the matching of characters. This provided strong evidence for the conjecture before \cite{ACL} appeared. Theorem \ref{ACLmain} immediately implies the unitarity of the discrete series $\W$-algebras. It also has the following striking application. Suppose that $\g$ is simply laced and $k$ is an admissible level for $\hat{\g}$. In \cite{CHY,C}, a certain subcategory of the Bernstein-Gelfand-Gelfand category $\mathcal{O}$ for the corresponding affine Lie algebra $\hat{\g}$ was considered. Using Theorem \ref{ACLmain}, it was shown that this category has a natural ribbon vertex tensor category structure, which is modular under some mild arithmetic conditions on $k$. This gives the first examples of such modular tensor categories for vertex algebras that are not lisse.

Since Theorem \ref{ACLmain} has many applications, it is an important problem to find coset realizations of other families of $\W$-(super)algebras. One such family involves the $\W$-superalgebra $\W^k(\slf(n+1|n))$ associated to the principal nilpotent element in the even part of $\slf(n+1|n)$. It was first suggested in the physics literature by Ito \cite{I1, I} that for generic values of $k$, this should be isomorphic to a certain coset vertex algebra called the $SU(n+1) / SU(n) \times U(1)$ coset model. It is an example of a general family of cosets constructed by Kazama and Suzuki in \cite{KS} which have $N=2$ superconformal symmetry. We give a slightly different construction of this coset which we denote by $C^l(n)$, where $k$ and $l$ are related by $(k + 1) (l + n + 1) = 1$; see Lemma \ref{ks:alternative}. In this paper, we shall call this expected isomorphism {\it Ito's conjecture}; it was previously stated in this form in the introduction of \cite{CL}. It clearly holds in the case $n=1$, since both $\W^k(\slf(2|1))$ and $C^l(1)$ are known to be isomorphic to the $N=2$ superconformal vertex algebra \cite{KRW,CL}. Our main result in this paper is that Ito's conjecture holds in the first nontrivial case $n=2$. As a corollary, we show that the simple $\W$-superalgebras $\W_k(\slf(3|2))$ are lisse and rational, where $(k+1)(l+3)=1$ and $l>1$ is a positive integer.

\section{Free field realization of $\W^k(\slf(n+1|n))$}\label{Introduction sec}
We introduce free field realizations of the $\W$-superalgebra $\W^k(\slf(n+1|n))$ associated to the principal nilpotent element in the even part of $\slf(n+1|n)$. In this paper, we follow the same notations for vertex (super)algebras as in Section 2 in \cite{CL}. Let $\{e_{i,j}\}_{i,j=1}^{2n+1}$ be the standard basis of $\gl(n+1|n)$. Denote by $\widetilde{\h}=\bigoplus_{i=1}^{2n+1}\C e_{i,i}$ a Cartan subalgebra of $\gl(n+1|n)$ and by $\widetilde{\h}^*=\bigoplus_{i=1}^{2n+1}\C\epsilon_i$ the dual of $\widetilde{\h}$, where $\{\epsilon_i\}_{i=1}^{2n+1}$ is the dual basis of $\widetilde{\h}^*$, i.e. $\epsilon_i(e_{j,j})=\delta_{i,j}$. Let $\h=\{ h\in\widetilde{\h}\mid\str(h)=0\}$ be a Cartan subalgebra of $\slf(n+1|n)\subset\gl(n+1|n)$, where $\str$ denotes the super trace. We identify $\h^*$ with $\h$ by the super trace and denote by $\h^*\ni\lambda\mapsto t_\lambda\in\h$, i.e. $\lambda(h)=\str(t_\lambda h)$ for all $h\in\h$. Set a non-degenerate symmetric bilinear form $(\lambda|\lambda')=\str(t_\lambda t_{\lambda'})$ on $\h^*$ for $\lambda,\lambda'\in\h^*$. Then $\Delta=\{\epsilon_i-\epsilon_j\mid1\leq i\neq j\leq2n+1\}$ is the root system of $\slf(n+1|n)$ associated with $\h$. Define a set $\Pi=\{\alpha_i\mid i=1,\ldots,2n\}$ of simple roots by $\alpha_{2i-1}=\epsilon_i-\epsilon_{i+n+1}$ and $\alpha_{2i}=\epsilon_{i+n+1}-\epsilon_{i+1}$ for $i=1,\ldots,n$. Since $(\alpha_i|\alpha_j)=(-1)^{i+1}\delta_{j,i+1}$ for all $i\leq j$, all simple roots are odd and isotropic. Fix a root vector $e_{\epsilon_i-\epsilon_j}=e_{i,j}$ for $i\neq j$. Let
\begin{align*}
f=f_\prin:=\sum_{i=1}^{2n-1}e_{-\alpha_i-\alpha_{i+1}}=\sum_{\begin{subarray}{c} 1\leq i \leq 2n \\ i\neq n+1\end{subarray}}e_{i+1,i}
\end{align*}
be a principal nilpotent element in the even part of $\slf(n+1|n)$, and
\begin{align*}
x:=\sum_{i=1}^{n}\frac{i}{2}(t_{\alpha_{2i}}+t_{\alpha_{2n-2i+1}})=\sum_{i=1}^{n+1}\left(\frac{n}{2}-i+1\right)e_{i,i}+\sum_{i=1}^{n}\left(\frac{n}{2}-i+\frac{1}{2}\right)e_{i+n+1,i+n+1}
\end{align*}
a semisimple element of $\slf(n+1|n)$ in $\h$. Then $\ad(x)$ defines a good grading $\Gamma=\Gamma_x$ on $\slf(n+1|n)$ for $f_\prin$ such that all positive roots has non-negative degree. The corresponding weighted Dynkin diagram is the following:
\vspace{1mm}
\begin{align*}
\hspace{30mm}\\
\setlength{\unitlength}{1mm}
\begin{picture}(0, 0)(20,10)
\put(-31.5,10){\circle{2}}
\put(-32.6,9.3){\footnotesize$\times$}
\put(-32.8,14){\footnotesize$\frac{1}{2}$}
\put(-32.5,5){\footnotesize$\alpha_1$}
\put(-30.5,10.3){\line(1,0){8}}
\put(-21.5,10){\circle{2}}
\put(-22.6,9.3){\footnotesize$\times$}
\put(-22.8,14){\footnotesize$\frac{1}{2}$}
\put(-22.5,5){\footnotesize$\alpha_2$}
\put(-20.5,10.3){\line(1,0){6}}
\put(-13,9.4){$\cdot$}
\put(-11.5,9.4){$\cdot$}
\put(-10,9.4){$\cdot$}
\put(-7,10.3){\line(1,0){6}}
\put(0,10){\circle{2}}
\put(-1.1,9.3){\footnotesize$\times$}
\put(-1.3,14){\footnotesize$\frac{1}{2}$}
\put(-0.8,5){\footnotesize$\alpha_{2n-1}$}
\put(1,10.3){\line(1,0){8}}
\put(10,10){\circle{2}}
\put(8.9,9.3){\footnotesize$\times$}
\put(8.7,14){\footnotesize$\frac{1}{2}$}
\put(7.9,5){\footnotesize$\alpha_{2n}$}
\put(12,9){.}
\end{picture}\\
\hspace{60mm}
\end{align*}
\vspace{2mm}\\
See \cite{Hoyt} for good gradings of $\slf(n+1|n)$. We have $\g_0=\h$ and $\tau_k(u|v)=(k+1)\str(uv)$ for all $u,v\in\h$. Thus, $\Hc=\Vtau$ is the Heisenberg vertex algebra generated by fields $\alpha_i(z)$ of conformal degree $1$ for $i=1,\ldots,2n$ which satisfy that
\begin{align*}
\alpha_i(z)\alpha_j(w)\sim\frac{(-1)^{i+1}(k+1)\delta_{j,i+1}}{(z-w)^2}
\end{align*}
for all $i\leq j$. Since $\g_{\frac{1}{2}}=\bigoplus_{i=1}^{2n}\C e_{\alpha_i}$ and $\str(f[e_{\alpha_i},e_{\alpha_j}])=(-1)^{i+1}\delta_{j,i+1}$ for all $i\leq j$, $\Fne$ is the vertex superalgebra generated by odd fields $\Phi_i(z)=\Phi_{\alpha_i}(z)$ of conformal degree $\frac{1}{2}$ which satisfy that
\begin{align*}
\Phi_i(z)\Phi_j(w)\sim\frac{(-1)^{i+1}\delta_{j,i+1}}{z-w}
\end{align*}
for all $i\leq j$. By \cite{G}, for generic $k$, we have an isomorphism
\begin{align}\label{scr eq}
\W^k(\slf(n+1|n)):=\W^k(\slf(n+1|n),f_\prin;\Gamma_x)\simeq\bigcap_{i=1}^{2n}\Ker\scr_i,
\end{align}
where
\begin{align*}
\scr_i=\int:\e^{-\frac{1}{k+1}\int\alpha_i(z)}\Phi_i(z):dz
\end{align*}
for all $i$ are screening operators acting on $\Hc\otimes\Fne$. If $k\neq-1$, $\W^k(\slf(n+1|n))$ is conformal with central charge $-3n(kn+k+n)$.

Replacing $k$ by a indeterminate $\bk$ in $T=\C[\bk,(\bk+1)^{-\frac{1}{2}}]$, we define the $\W$-algebra $\W^\bk(\slf(n+1|n))$ over $T$, which also satisfies \eqref{scr eq}. Let
\begin{align*}
G_+^{(n)}&=\frac{1}{\sqrt{\bk+1}}\sum_{i=1}^n\left(\sum_{j=i}^{n}:\alpha_{2i-1}\Phi_{2j}:+i(\bk+1)\der\Phi_{2i}\right),\\
G_-^{(n)}&=\frac{1}{\sqrt{\bk+1}}\sum_{i=1}^n\left(\sum_{j=1}^{i}:\alpha_{2i}\Phi_{2j-1}:+(n-i+1)(\bk+1)\der\Phi_{2i-1}\right),\\
H^{(n)}&=(G_+^{(n)})_{(1)}(G_-^{(n)}),\quad
L^{(n)}=(G_+^{(n)})_{(0)}(G_-^{(n)})-\frac{1}{2}H^{(n)}.
\end{align*}
Then $H^{(n)}(z),G_\pm^{(n)}(z),L^{(n)}(z)$ generate a copy of $N=2$ superconformal vertex algebra, and belong to $\bigcap_{i=1}^{2n}\Ker\scr_i$, which has been shown in \cite[Section 2.2]{I}. Since the specialization of $\W^\bk(\slf(n+1|n))$ at $\bk=k\in\C\backslash\{-1\}$ coincides with $\W^k(\slf(n+1|n))$, $\W^k(\slf(n+1|n))$ has a copy of $N=2$ algebra. As shown in Section 7 in \cite{KRW}, in the case $n=1$, $\W^k(\slf(2|1))$ is just isomorphic to the $N=2$ algebra.

\section{The Kazama-Suzuki coset}
In \cite{KS}, Kazama and Suzuki constructed a class of coset vertex algebras called $G/H$ models that are attached to a Lie group $G$ and a closed subgroup $H \subset G$. Let $\g$ and $\h$ be the Lie algebras of $G$ and $H$, respectively, and let $\F(\g)$ and $\F(\h)$ denote the free fermion algebras attached to $\g$ and $\h$. If $\g$ is simple, there is a homomorphism $V^{h^{\vee}}(\g) \rightarrow \F(\g)$ sending $u(z)\mapsto \tilde{u}(z)$, so there is a diagonal homomorphism $V^{k}(\g) \rightarrow V^{k-h^{\vee}}(\g) \otimes \F(\g)$. Replacing the generators $u(z)$ of $V^{k-h^{\vee}}(\g)$ with the diagonal generators $u(z)\otimes 1 + 1 \otimes \tilde{u}(z)$, we see that $V^{k-h^{\vee}}(\g) \otimes \F(\g)$ is isomorphic to the vertex superalgebra with OPE relations given by (2.4)-(2.6) of \cite{KS}. We denote this vertex superalgebra by $SV^k(\g)$, and we observe that both $V^k(\g)$ and  $\F(\g)$ are subalgebras of $SV^k(\g)$. Additionally, $SV^k(\g)$ has an action of the $N=1$ superconformal vertex algebra given by (2.23)-(2.24) of \cite{KS}. In fact, this construction makes sense for any finite-dimensional $\g$ with a nondegenerate symmetric invariant bilinear form, not necessarily simple.

The embedding $\h \hookrightarrow \g$ induces a homomorphism $SV^{k'}(\h) \rightarrow SV^{k}(\g)$, where the relationship between the levels $k,k'$ is determined from this embedding, and the $G/H$ coset model is just the coset $\text{Com}(SV^{k'}(\h), SV^k(\g))$. These cosets have an $N=1$ superconformal symmetry which is given explicitly by (2.37) and (2.42) of \cite{KS}, and the authors also determined the precise conditions under which this is enhanced to $N=2$ symmetry.

The relevant case here corresponds to $G = SU(n+1)$ and $H = SU(n) \times U(1)$, which has $N=2$ symmetry. In this case, there is a slightly simpler construction of the same coset, which we describe below; see also Example 7.11 of \cite{CL}. Let $\E$ be the $bc$-system, i.e., the vertex superalgebra generated by odd fields $b(z)$ and $c(z)$ which satisfy
\begin{align*}
b(z)c(w)\sim\frac{1}{z-w},\quad
b(z)b(w)\sim 0 \sim c(z)c(w).
\end{align*} Set $\E(n)=\E^{\otimes n}$, and denote by $b_i(z), c_i(z)$ the generating fields $b(z), c(z)$ of the $i$-th component in $\E(n)$. Then $b_i(z), c_i(z)$ are both primary of conformal degree $\frac{1}{2}$ with respect to the Virasoro element
$$L(z) = -\frac{1}{2}\sum_{i=1}^n \big( :b_i(z) \partial c_i(z): - :(\partial b_i(z))c_i(z):\big),$$ which has central charge $n$.
We have a $V^1(\gl_n)$-module structure on $\E(n)$ defined by $$e_{i,j}(z)\mapsto \ :b_i(z)c_j(z):,$$ which in fact descends to an action of the simple quotient $L_1(\gl_n)$. Here $\{ e_{i,j}\}_{i,j=1}^n$ denotes the standard basis of $\gl_n$. Set $h_i=e_{i,i}-e_{i+1,i+1}$. Since $V^l(\gl_n)\subset V^l(\slf_{n+1})$ by
\begin{equation}\label{eq:embedding}
\begin{split}
e_{i,j}(z)&\mapsto e_{i,j}(z)\quad(i\neq j),\quad
h_i(z)\mapsto h_i(z),\quad (i=1,\ldots,n-1),
\\ \sum_{i=1}^{n+1}e_{i,i}(z)&\mapsto \varpi_n(z)=\frac{1}{n+1}\sum_{i=1}^n i h_i(z),
\end{split}
\end{equation}
we have an embedding $V^{l+1}(\gl_n)\hookrightarrow V^l(\slf_{n+1})\otimes\E(n)$ defined by the diagonal action, where $\varpi_n$ is a semisimple element corresponding to the $n$-th fundamental weight of $\slf_{n+1}$.

\begin{lemma} \label{lem:commtensor} For any vertex (super)algebras $\mathcal{A}$, $\mathcal{B}$, and $\mathcal{C}$ with $\mathcal{A} \subset \mathcal{B}$, we have 
\begin{equation*} \text{Com}(\mathcal{A}\otimes \mathcal{C}, \mathcal{B} \otimes \mathcal{C}) = \text{Com}(\mathcal{A}, \mathcal{B}) \otimes \text{Com}(\mathcal{C}, \mathcal{C}).\end{equation*} In particular, $\text{Com}(\mathcal{A}\otimes \mathcal{C}, \mathcal{B} \otimes \mathcal{C}) = \text{Com}(\mathcal{A}, \mathcal{B})$ whenever $\mathcal{C}$ has trivial center.
\end{lemma}

\begin{proof} The inclusion $ \text{Com}(\mathcal{A}, \mathcal{B}) \otimes \text{Com}(\mathcal{C}, \mathcal{C}) \subset \text{Com}(\mathcal{A}\otimes \mathcal{C}, \mathcal{B} \otimes \mathcal{C})$ is obvious, so let $ \omega \in \text{Com}(\mathcal{A}\otimes \mathcal{C}, \mathcal{B} \otimes \mathcal{C})$. Without loss of generality, we may write \begin{equation*} \omega = \sum_i b_i \otimes c_i,\end{equation*} where $\{b_i\}$ is a linearly independent subset of $\mathcal{B}$, and $c_i \in \mathcal{C}$. Since $\omega$ commutes with $\mathcal{C}$ and each $b_i$ commutes with $\mathcal{C}$, it follows from the linear independence of $\{b_i\}$ that each $c_i \in \text{Com}(\mathcal{C},\mathcal{C})$.

Next, let $\{c'_j\}$ be any basis for $\text{Span}(\{c_i\})$, and rewrite $\omega$ in the form 
\begin{equation*} \omega = \sum_j b'_j \otimes c'_j, \end{equation*} where $b'_j$ are appropriate linear combinations of $\{b_i\}$. Clearly each $c'_j \in \text{Com}(\mathcal{C},\mathcal{C})$. Since $\omega$ commutes with $\mathcal{A}$ and each $c'_j$ commutes with $\mathcal{A}$, it follows from the linear independence of $\{c'_j\}$ that each $b'_j \in \text{Com}(\mathcal{A},\mathcal{B})$. Therefore $\omega \in  \text{Com}(\mathcal{A}, \mathcal{B}) \otimes \text{Com}(\mathcal{C}, \mathcal{C})$.
\end{proof}

\begin{lemma} \label{ks:alternative} The coset vertex superalgebra
\begin{align*}
C^l(n):=\Com(V^{l+1}(\gl_n), V^l(\slf_{n+1})\otimes\E(n)),
\end{align*}
which has central charge $\displaystyle c = \frac{3n l}{1+n+l}$, is isomorphic to the $SU(n+1)/ SU(n) \times U(1)$ coset $\text{Com}(SV^{k'}(\gl_n), SV^k(\slf_{n+1}))$ in  \cite{KS}, where $k' = k$ and $l = k - n - 1$. 
\end{lemma}

\begin{proof} First, recall that we can replace $SV^k(\gl_{n})$ and $SV^k(\slf_{n+1})$ with $V^{k-n}(\gl_{n}) \otimes \F(\gl_{n})$ and $V^{k-n-1}(\slf_{n+1}) \otimes \F(\slf_{n+1})$, respectively. Using the orthogonal decomposition $\slf_{n+1} = \gl_n \oplus \slf_{n+1}/\gl_n$, we have $$SV^k(\slf_{n+1}) \cong  V^{k-n-1}(\slf_{n+1}) \otimes \F(\gl_{n}) \otimes \F(\slf_{n+1}/\gl_{n}).$$ By (4.3) of \cite{KS}, the fermion algebra $\F(\slf_{n+1}/\gl_{n})$ is isomorphic to $\E(n)$, and the map $SV^{k}(\gl_n) \rightarrow SV^k(\slf_{n+1})$ clearly restricts to the inclusion $\F(\gl_n) \hookrightarrow \F(\gl_{n}) \otimes \E(n)$ sending $a \mapsto a \otimes 1$.

We now apply Lemma \ref{lem:commtensor} to the case $\mathcal{A}=V^{l+1}(\gl_n)$, $\mathcal{B}=V^l(\slf_{n+1})\otimes\E(n)$ and $\mathcal{C}=\F(\gl_{n})$. Since $\F(\gl_{n})$ has trivial center, it follows that
\begin{equation*} \begin{split} & \text{Com}(SV^{k}(\gl_n), SV^k(\slf_{n+1})) = \text{Com}(V^{l+1}(\gl_n) \otimes \F(\gl_{n}), V^l(\slf_{n+1})\otimes\E(n) \otimes \F(\gl_{n})) 
\\ & \cong \text{Com}(V^{l+1}(\gl_n), V^l(\slf_{n+1})\otimes\E(n)) = C^l(n).\end{split} \end{equation*} \end{proof}

In view of this isomorphism, we shall call $C^l(n)$ the {\it Kazama-Suzuki coset}. It is expected to be isomorphic to $\W^k(\slf(n+1|n))$ for generic $k$, where $(k + 1) (l + n + 1) = 1$ \cite{I1}. This conjecture was first stated in this form in the introduction of \cite{CL}, and shall be called {\it Ito's conjecture} in this paper. Here we discuss some features of $C^l(n)$ that hold for all $n\geq 3$. As shown in Example 7.11 of \cite{CL}, $C^l(n)$ has a minimal strong generating set consisting of even fields in conformal degrees $$1,2,2,3,3,\dots, n,n,n+1,$$ and odd fields in conformal degrees $$\frac{3}{2}, \frac{3}{2}, \frac{5}{2}, \frac{5}{2}, \dots, \frac{2n+1}{2}, \frac{2n+1}{2}.$$ Also, the Heisenberg field, the Virasoro field, and two odd fields in conformal degree $\frac{3}{2}$ generate a copy of the $N=2$ superconformal vertex algebra. In the case $n=1$, $C^l(1)$ is just the $N=2$ algebra; see Lemma 8.6 of \cite{CL}. 

\begin{lemma} For $n\geq 2$, and generic values of $l$, we have a conformal embedding
\begin{equation} \label{eq:confemb} \mathcal{H}_0 \otimes \W^{r}(\slf_n) \otimes \mathcal{G}^l(n) \hookrightarrow C^l(n),\qquad r= -n + \frac{l + n}{1 + l + n}.\end{equation}
Here $\mathcal{H}_0$ is the rank one Heisenberg algebra, $\W^{r}(\slf_n)$ is the principal $\W$-algebra of $\slf_n$ at level $r$, and 
$$\mathcal{G}^l(n) = \Com(V^{l}(\gl_n), V^l(\slf_{n+1})),$$ which was called a generalized parafermion algebra in \cite{L}.
\end{lemma}

\begin{proof} Since $V^l(\gl_n)$ embeds in $V^l(\slf_{n+1})$ via \eqref{eq:embedding}, $C^l(n)$ contains a copy of the coset $\Com(V^{l+1}(\gl_n), V^l(\gl_n)\otimes\E(n))$, which is isomorphic to
$$\mathcal{H}_0 \otimes \Com(V^{l+1}(\slf_n), V^l(\slf_n)\otimes\E(n)) \cong \mathcal{H}_0 \otimes \Com(V^{l+1}(\slf_n), V^l(\slf_n)\otimes L_1(\slf_n )).$$ By Theorem \ref{ACLmain}, this coset is isomorphic to $\mathcal{H}_0 \otimes \W^{r}(\slf_n)$. 

It is also clear from the definition of $C^l(n)$ that it contains a copy of $\mathcal{G}^l(n)$; what remains to show is that this copy of $\mathcal{G}^l(n)$ commutes with $\mathcal{H}_0 \otimes \W^{r}(\slf_n)$ and that \eqref{eq:confemb} is a conformal embedding. But this is clear from Theorems 5.1 and 5.2 of \cite{FZ}, since the Virasoro elements for the cosets $\Com(\mathcal{H}_0, C^l(n))$, $\Com(\W^r(\slf_n), C^l(n))$, and $\Com(\mathcal{G}^l(n), C^l(n))$ pairwise commute, and their sum is the total Virasoro element for $C^l(n)$.
\end{proof}

\begin{rem} By Theorem 8.1 of \cite{L}, $\mathcal{G}^l(n)$ is of type $\W(2,3,\dots, n^2+3n+1)$ for generic values of $l$. In the case $n = 1$, $\mathcal{G}^l(1)$ coincides with the parafermion algebra of $\slf_2$, which is of type $\W(2,3,4,5)$ \cite{DLY}.
\end{rem}

Let $C_l(n)$ denote the unique simple graded quotient of $C^l(n)$. Suppose first that $l>1$ is a positive integer. Then the maps $$V^{l}(\gl_n) \hookrightarrow V^l(\slf_{n+1}),\qquad V^{l+1}(\gl_n) \hookrightarrow V^l(\slf_{n+1}) \otimes \E(n)$$ induce maps of simple vertex algebras
\begin{equation} \label{eq:simplequotientmap} L_{l}(\slf_n) \otimes \mathcal{H}_0\hookrightarrow L_l(\slf_{n+1}),\qquad  L_{l+1}(\slf_n) \otimes \mathcal{H}_0\hookrightarrow L_l(\slf_{n+1}) \otimes \E(n).\end{equation}
In fact, these maps extend to embeddings

\begin{equation} \begin{split} & L_{l}(\slf_n) \otimes V_L \hookrightarrow L_l(\slf_{n+1}),
\\ &  L_{l+1}(\slf_n) \otimes V_L \hookrightarrow L_l(\slf_{n+1}) \otimes \E(n).\end{split} \end{equation}  Here $V_L$ is the lattice vertex algebra associated to the rank one lattice $$L = \sqrt{n(n+1)(n+l+1)}\ \mathbb{Z},$$ which is an extension of $\mathcal{H}_0$. It follows from Theorem 8.1 of \cite{CL} that the simple quotient $\mathcal{G}_l(n)$ of $\mathcal{G}^l(n)$ coincides with the coset
$$\Com(L_{l}(\slf_n) \otimes V_L , L_l(\slf_{n+1})).$$ Similarly,
$$C_l(n) = \Com(L_{l+1}(\slf_n) \otimes V_L , L_l(\slf_{n+1}) \otimes \E(n)).$$
In particular, $C_l(n)$ is a vertex superalgebra extension of
$$V_L  \otimes \W_r(\slf_n) \otimes \mathcal{G}_l(n).$$ It was shown in \cite{ACL} that 
$$\mathcal{G}_l(n) \cong \W_s(\slf_l),\qquad r= -n + \frac{l + n}{1 + l + n},\qquad s = -l + \frac{l+n}{1+l+n}.$$ which is lisse and rational by \cite{Ar2,Ar3}. It follows that $C_l(n)$ is a rational vertex superalgebra; see Corollary 14.1 of \cite{ACL}. 

\begin{rem} The case $l=1$ is degenerate because $L_{1}(\slf_n) \otimes \mathcal{H}_0$ is conformally embedded in $L_1(\slf_{n+1})$ \cite{AKMPP}, so that $\mathcal{G}_1(n) \cong \mathbb{C}$. In this case, $C_1(n)$ is an extension of $V_L  \otimes \W_r(\slf_n)$, where $L = \sqrt{n(n+1)(n+2)}\ \mathbb{Z}$ and $\displaystyle r =  -n + \frac{1 + n}{2 + n}$, so it is also rational and lisse.  
\end{rem}

For $l >1$, the key ingredient in proving the rationality of $C_l(n)$ is that $\mathcal{G}_l(n)$ is coincident with a type $A$ principal $\W$-algebra. In fact, the levels $l$ where $\mathcal{G}_l(n)$ is isomorphic to $\W_{s}(\slf_m)$ for some $m \geq 3$ and some level $s \in \mathbb{C}$, have been classified; see Theorem 10.12 of \cite{L}. In addition to the above family when $l = m$, we have the following two additional families:

\begin{enumerate}
\item For $n> 1$, $m\geq 3$, $m \neq n+1$, and $\displaystyle l =  -(n+1) + \frac{m - (n+1)}{m-1}$, we have  
$$\mathcal{G}_l (n) \cong \W_{s}(\slf_m),\qquad s = -m + \frac{m - n -1}{m -1},\ -m + \frac{m-1}{m - n -1},$$ which has central charge $\displaystyle c = \frac{(1 + n - m + n m) (m + n m-1)}{1 + n - m}$. Note that the level $l$ of $V^l(\slf_{n+1})$ is admissible if $m \geq 2(n+1)$, and $m-1$ and $m - (n+1)$ are relatively prime.

\item For $n>1$, $m\geq 3$, $m\neq n$, and $\displaystyle l = -(n+1) + \frac{1 + n}{1 + m}$, we have 
$$\mathcal{G}_l(n) \cong \W_{s} (\slf_m), \qquad s = -m + \frac{m -n}{1 + m},\  -m + \frac{1 + m}{m -n},$$ which has central charge $\displaystyle c = \frac{n(m-1) (1 + 2 m + n m)}{n - m}$. Note that if $m+1$ and $n+1$ are relatively prime the level $l$ is boundary admissible for $\slf_{n+1}$
\end{enumerate}

\begin{ques} For the above levels $l$, do the maps $$V^{l}(\gl_n) \hookrightarrow V^l(\slf_{n+1}),\qquad V^{l+1}(\gl_n) \hookrightarrow V^l(\slf_{n+1}) \otimes \E(n)$$ induce maps of simple vertex algebras
\begin{equation} \label{eq:simplequotientmap} L_{l}(\slf_n) \otimes \mathcal{H}_0\hookrightarrow L_l(\slf_{n+1}), \qquad L_{l+1}(\slf_n) \otimes \mathcal{H}_0\hookrightarrow L_l(\slf_{n+1}) \otimes \E(n)?\end{equation} 
\end{ques}

\begin{rem} If the maps \eqref{eq:simplequotientmap} exist, Theorem 8.1 of \cite{CL} would imply that
$$\mathcal{G}_l(n) = \Com(L_{l}(\slf_n) \otimes \mathcal{H}_0, L_l(\slf_{n+1})),\qquad C_l(n) = \Com(L_{l+1}(\slf_n) \otimes \mathcal{H}_0, L_l(\slf_{n+1}) \otimes \E(n)),$$ and we would have a conformal embedding
$$ \mathcal{H}_0 \otimes \W_{r} (\slf_n) \otimes \W_{s}(\slf_m) \hookrightarrow C_l(n),\qquad r = -n + \frac{l+n}{1+l+n}.$$ 
\end{rem}

\section{Ito's conjecture for $n=2$}
In this section, we shall prove our main result, which is that Ito's conjecture holds for $n=2$. For this, we need the explicit free field realization of $\W^k(\slf(3|2))$. Recall that $\W^k(\slf(3|2))$ is isomorphic to the intersection of the kernels of the screening operators $\scr_i$ acting on $\Hc\otimes\Fne$ for $i=1,\ldots,4$ if $k$ is generic, see Section \ref{Introduction sec}. We introduce two even fields $H,S$ and two odd fields $G_\pm$ on $\Hc\otimes\Fne$ defined as follows:
\begin{align*}
&H=2\alpha_1-\alpha_2+\alpha_3-2\alpha_4-:\Phi_1\Phi_2:-:\Phi_1\Phi_4:-:\Phi_3\Phi_4:,\\
&G_+=:\alpha_1\Phi_2:+:\alpha_1\Phi_4:+:\alpha_3\Phi_4:+(k+1)\der\Phi_2+2(k+1)\der\Phi_4,\\
&G_-=:\alpha_2\Phi_1:+:\alpha_4\Phi_1:+:\alpha_4\Phi_3:+2(k+1)\der\Phi_1+(k+1)\der\Phi_3,\\
&S=\frac{3}{2}\Bigl(:\alpha_1^2:+:\alpha_1\alpha_3:-:\alpha_1\alpha_4:-\frac{1}{2}:\alpha_2^2:-2:\alpha_2\alpha_3:+:\alpha_2\alpha_4:-\frac{1}{2}:\alpha_3^2:\\
&+:\alpha_4^2:-:\alpha_1(\Phi_1\Phi_2+\Phi_1\Phi_4+\Phi_3\Phi_4):-:\alpha_2(\Phi_1\Phi_2+\Phi_1\Phi_4-2\Phi_3\Phi_4):\\
&-:\alpha_3(2\Phi_1\Phi_2-\Phi_1\Phi_4-\Phi_3\Phi_4):+:\alpha_4(\Phi_1\Phi_2+\Phi_1\Phi_4+\Phi_3\Phi_4):+2:\Phi_1\Phi_2\Phi_3\Phi_4:\\
&-\frac{1}{2}(4k+3):\Phi_1\der\Phi_2:+\frac{1}{2}(2k+3):\Phi_1\der\Phi_4:+\frac{1}{2}(2k+3):\Phi_3\der\Phi_4:\\
&-\frac{1}{2}(2k+3):(\der\Phi_1)\Phi_2:-\frac{1}{2}(2k+3):(\der\Phi_1)\Phi_4:+\frac{1}{2}(4k+3):(\der\Phi_3)\Phi_4:\\
&+(k+1)\der(\alpha_1-\alpha_2-\alpha_3+\alpha_4)\Bigr).
\end{align*}

\begin{lemma}\label{kernel lemma}
$H, S, G_\pm$ belong to $\bigcap_{i=1}^{4}\Ker\scr_i$.
\end{lemma}
\begin{proof}
Direct calculations.
\end{proof}

Set
\begin{align*}
&L=\frac{1}{k+1}{G_+}_{(0)}G_{-}-\frac{1}{2}\der H,\quad
W_2=\frac{1}{2}L+S,\quad
Q_{\pm}={G_\pm}_{(0)}S+\frac{1}{4}\der G_\pm,\\
&W_3={G_+}_{(0)}Q_{-}-\frac{1}{4}(k+1)\left(2\der S+\der L+6:H L:-2:H^3:\right),
\end{align*}
where $k\neq-1$, and we denote by $A(z)=\sum_{n\in\Z}A_{(n)}z^{-n-1}$ the field corresponding to $A$ for any $A\in\Hc\otimes\Fne$. Then $L$ is a Virasoro element of $\Hc\otimes\Fne$ with central charge $-6(3k+2)$. Even primary elements $H, W_2, W_3$ have conformal degree $1,2,3$, and odd primary elements $G_+, G_-, Q_+, Q_-$ have conformal degree $\frac{3}{2}, \frac{3}{2}, \frac{5}{2}, \frac{5}{2}$ respectively.

\begin{prop}\label{gen prop}
For generic $k$, the elements $H, L, W_2, W_3, G_\pm, Q_\pm$ strongly generate the $\W$-algebra $\W^k(\slf(3|2))$ in $\Hc\otimes\Fne$.
\end{prop}
\begin{proof}
In general, by \cite[Theorem 4.1]{KW3}, if $\{ u_i\}_{i=1}^{\dim\g^f}$ is a basis of $\g^f$ with $u_i\in\g_{j_i}$, where $\g^f$ is the centralizer of $f$ in $\g$, the $\W$-algebra $\W^k(\g,f;\Gamma)$ has a minimal set $\{ \widetilde{J}_{u_i}\}_{i=1}^{\dim\g^f}$ of strong generators, and then $\widetilde{J}_{u_i}$ has the conformal degree $1-j_i$ and the same parity as $u_i$.

For generic $k$, let $\V=\bigcap_{i=1}^{4}\Ker\scr_i$ be a vertex subalgebra of $\Hc\otimes\Fne$. By \eqref{scr eq}, $\V$ is isomorphic to $\W^k(\slf(3|2))$. By Lemma \ref{kernel lemma}, all elements $H$, $L$, $W_2$, $W_3$, $G_\pm$, $Q_\pm,$ belong to $\V$. Since $\dim\slf(3|2)^{f_\prin}=8$, and $\dim\slf(3|2)^{f_\prin}\cap\slf(3|2)_j$ is equal to $1$ for $j=0,-2$ and is equal to $2$ for $j=-\frac{1}{2},-1,-\frac{3}{2}$, $\V$ has a minimal strong generating set consisting of four even fields in conformal degrees $1,2,2,3$ and four odd fields in conformal degrees $\frac{3}{2}$, $\frac{3}{2}$, $\frac{5}{2}$, $\frac{5}{2}$ respectively. Let $\V_j$ be the subspace of $\V$ with conformal degree $j$ and $\mathcal{U}_j=\V_{<j}\cap\operatorname{Span}\{A_{1 (-n_1)}\cdots A_{s-1 (-n_{s-1})}A_{s}\in\V\mid A_{i}\in \V_{<j}, n_i\geq1, s\in\Z_{\geq1}\}$, where $\V_{<j}=\bigoplus_{p<j}\V_p$. Then $\dim\V_{<0}=\dim\V_{\frac{1}{2}}=0$, $\dim\V_0=1$, $\V_1=\C H$, $\V_{\frac{3}{2}}=\C G_+\oplus\C G_-$, $\V_2=\C L\oplus\C W_2$, $\V_{\frac{5}{2}}=\C Q_+\oplus\C Q_-\oplus\mathcal{U}_{\frac{5}{2}}$ and $\V_3=\C W_3\oplus\mathcal{U}_3$. Therefore $H, L, W_2, W_3, G_\pm, Q_\pm$ strongly generate $\V\simeq\W^k(\slf(3|2))$.
\end{proof}

\begin{cor}\label{gen cor}
If $k$ is generic, the vertex subalgebra of $\Hc\otimes\Fne$ (weakly) generated by $H,S,G_\pm$ is isomorphic to $\W^k(\slf(3|2))$.
\end{cor}
\begin{proof}
By definition, all elements $L, W_2, W_3, Q_\pm$ belong to the vertex subalgebra generated by $H,S,G_\pm$. Therefore the assertion follows from Proposition \ref{gen prop}.
\end{proof}

Next, we consider the coset $C^l(2)$.
Suppose that $l\neq-3$. Let $\widehat{H}, \widehat{S}$ be even elements and $\widehat{G}_\pm$ odd elements in $V^l(\slf_{3})\otimes\E(2)$ defined as follows:
\begin{align*}
&\widehat{H}=\frac{1}{l+3}\left(h_1-2h_2 +l:b_1 c_1:+l:b_2 c_2:\right),\\
&\widehat{G}_+=\frac{1}{l+3}\left(:e_{3,1} b_1:+e_{3,2} b_2:\right),\quad
\widehat{G}_-=\frac{1}{l+3}\left(:e_{1,3} c_1:+e_{2,3} c_2:\right),\\
&\widehat{S}=-\frac{3}{2(l+3)^2}\Bigl(\frac{1}{2}:h_1^2:+:h_1 h_2:-:h_2^2:+3:e_{1,2}e_{2,1}:-:e_{2,3}e_{3,2}:-:e_{1,3}e_{3,1}:\\
&-(2l+3):h_1 b_1 c_1:+(l+2):h_1 b_2 c_2:+(l+1):h_2 b_1 c_1:+(l+1):h_2 b_2 c_2:\\
&-(3l+5):e_{1,2}b_2 c_1:-(3l+5):e_{2,1}b_1 c_2:+l(2l+3):b_1 b_2 c_1 c_2:-\frac{1}{2}l(l+2):b_1 \der c_1:\\
&-\frac{1}{2}l(l+2):b_2\der c_2:+\frac{1}{2}l(l+2):(\der b_2)c_2:-2\der h_1+\der h_2\Bigr).
\end{align*}

\begin{lemma}\label{coset gen lem}
$\widehat{H}, \widehat{S},\widehat{G}_\pm$ belong to $C^l(2)$.
\end{lemma}
\begin{proof}
Direct calculations.
\end{proof}

For $l\neq-3$, set $k=\frac{1}{l+3}-1$. We define elements $\widehat{L},\widehat{W}_2,\widehat{W}_3,\widehat{Q}_\pm$ in $V^l(\slf_{3})\otimes\E(2)$ in the same way as $L,W_2,W_3,Q_\pm$, with $A$ replaced by $\widehat{A}$ for $A=H,S,G_\pm$. By Lemma \ref{coset gen lem}, all elements $\widehat{L},\widehat{W}_2,\widehat{W}_3,\widehat{Q}_\pm$ belong to $C^l(2)$.

\begin{lemma}
For generic $k$, there exists a vertex superalgebra homomorphism $\gamma\colon\W^k(\slf(3|2))\rightarrow C^l(2)$, where $(k+1)(l+3)=1$.
\end{lemma}
\begin{proof}
Suppose that $k$ is generic. By direct calculations, we have
\begin{align*}
&H(z)H(w)\sim\frac{-2(3k+2)}{(z-w)^2},\quad
H(z)S(w)\sim\frac{-\frac{3}{2}(2k+1)H(w)}{(z-w)^2},\quad
H(z)G_\pm(w)\sim\frac{\pm G_\pm(w)}{z-w},\\
&G_+(z)G_-(w)\sim{-2(k+1)(3k+2)}{(z-w)^3}+\frac{(k+1)H(w)}{(z-w)^2}+\frac{(k+1)(L(w)+\frac{1}{2}\der J(w)}{z-w},\\
&G_\pm(z)S(w)\sim\frac{-\frac{3}{4}G_\pm(w)}{(z-w)^2}+\frac{Q_\pm(w)-\frac{1}{4}\der G_\pm(w)}{z-w},\quad
G_+(z)G_+(w)\sim0\sim G_-(z)G_-(w),\\
&S(z)S(w)\sim\frac{\frac{9}{4}(3k+2)(12k^2+23k+6)}{(z-w)^4}\\
&+\frac{3(5k+2)S(w)-\frac{9}{2}(k+1)(4k+1)L(w)-\frac{9}{4}(3k+1):H(w)^2:}{(z-w)^2}\\
&+\frac{\frac{1}{2}\der\left(3(5k+2)S(w)-\frac{9}{2}(k+1)(4k+1)L(w)-\frac{9}{4}(3k+1):H(w)^2:\right)}{z-w}.
\end{align*}
Suppose that $(k+1)(l+3)=1$. By direct calculations, $\widehat{H}(z),\widehat{S}(z),\widehat{G}_\pm(z)$ also satisfy the same fomulae as above by replacing $A$ by $\widehat{A}$ for $A=H,S,G_\pm,L,Q_\pm$. By Proposition \ref{gen prop} and \cite[Theorem 4.1]{KW3}, the $\W$-algebra $\W^k(\slf(3|2))$ is freely generated by $H, L, W_2, W_3, G_\pm, Q_\pm$. Then, Corollary \ref{gen cor} and \cite[Proposition 3.11]{DSK} yields a well-defined vertex superalgebra homomorphism
\begin{align*}
\gamma\colon\W^k(\slf(3|2)) \rightarrow C^l(2),\quad
A\mapsto\widehat{A},\quad
A=H,S,G_\pm,
\end{align*}
as desired.
\end{proof}

\begin{theorem} \label{thm:n=2main}
The map $\gamma$ is an isomorphism. Therefore, for generic $k$, we have
\begin{align*}
\W^k(\slf(3|2))\simeq C^l(2),
\end{align*}
where $(k+1)(l+3)=1$.
\end{theorem}

\begin{proof} Since $\W^k(\slf(3|2))$ is generically simple, the map $\gamma$ is injective for generic $k$, so it suffices to show that the coset $C^l(2)$ is strongly generated by the corresponding fields $\widehat{H},\widehat{L}, \widehat{W}_2, \widehat{W}_3,   \widehat{G}_{\pm}, \widehat{Q}_{\pm}$ for generic values of $l$. 

If we rescale the generators of $V^l(\slf_{3})$ by $\frac{1}{\sqrt{l}}$, as explained in Example 3.1 of \cite{CL}, we have 
$$\lim_{l \rightarrow \infty} V^l(\slf_{3}) \cong \mathcal{H}(8),$$ where $\mathcal{H}(8)$ denotes the rank $8$ Heisenberg algebra. In particular, in the limit $l \rightarrow \infty$, the rescaled generators 
$$\tilde{e}_{1,3} = \frac{1}{\sqrt{l}} \ e_{1,3},\qquad \tilde{e}_{2,3} = \frac{1}{\sqrt{l}} \ e_{2,3}, \qquad \tilde{e}_{3,1} = \frac{1}{\sqrt{l}}\ e_{3,1}, \qquad \tilde{e}_{3,2} = \frac{1}{\sqrt{l}} \ e_{3,2}$$  satisfy the following nontrivial OPEs
$$ \tilde{e}_{1,3}(z) \tilde{e}_{3,1}(w) \sim (z-w)^{-2},\qquad  \tilde{e}_{2,3} (z) \tilde{e}_{3,2}(w) \sim (z-w)^{-2}.$$
As shown in Example 7.11 of \cite{CL} in the case $n=2$, $$\lim_{l \rightarrow \infty} C^l(2) \cong (\mathcal{H}(4) \otimes \E(2))^{GL_2}.$$ The generators of $\mathcal{H}(4)$ are denoted by $a^1, a^2, \bar{a}^1, \bar{a}^2$ in \cite{CL}, and they correspond respectively to the  $l \rightarrow \infty$ limits of $\tilde{e}_{1,3}, \tilde{e}_{2,3}, \tilde{e}_{3,1}, \tilde{e}_{3,2}$. Also, by Lemma 7.12 of \cite{CL}, the generators of $(\mathcal{H}(4) \otimes \E(2))^{GL_2}$ are given by 
$$j^k =\  :b_1 \partial^k c_1: + :b_2 \partial^k c_2:,\qquad w^k = \ :a^1 \partial^k \bar{a}^1: + :a^2 \partial^k \bar{a}^2:,\qquad k = 0,1,$$
$$\nu^k = \ :b_1 \partial^k \bar{a}^1: +  :b_2 \partial^k \bar{a}^2:,\qquad  \mu ^k = \ :a^1 \partial^k c_1: +  :a^2 \partial^k c_2:,\qquad k = 0,1.$$

Next, replace the generators $e_{i,j}, h_i$ of $V^l(\slf_3)$ with $\sqrt{l} \ \tilde{e}_{i,j}, \sqrt{l} \ \tilde{h}_i$, respectively, and consider the following fields in $V^l(\slf_3) \otimes \E(2)$, expressed in terms of the variables $\tilde{e}_{i,j}, \tilde{h}_i, b_i, c_i$.

\begin{equation} \label{eq:limitfields} \begin{split}  J^0 & =  \widehat{H},\qquad J^1 = \frac{2}{9}\ \widehat{W}_2 - \frac{1}{9}\ \widehat{L} - \frac{1}{3} :\widehat{H}\widehat{H}: + \frac{1}{2}\ \partial \widehat{H},
\\  \Omega^0 & = \frac{8}{9}\  \widehat{L} + \frac{2}{9} \ \widehat{W}_2 - \frac{1}{3} :\widehat{H}\widehat{H}: ,
\\  \Omega^1 & = \frac{2 l}{9}\ \widehat{W}_3  - \frac{2}{9}  :\widehat{H}\widehat{L}:  + \frac{1}{9} \ \partial \widehat{L}  - \frac{1}{3} :(\partial \widehat{H}) \widehat{H}: + \frac{1}{18} :\widehat{H}\widehat{H}\widehat{H}:   \\ &- \frac{1}{9}  \ \partial^2 \widehat{H}  + \frac{1}{9} \ \partial \widehat{W}_2 + \frac{2 l}{3} :\widehat{G}_+ \widehat{G}_-: - \frac{2}{9} :\widehat{H}\widehat{W}_2:,
 \\  N^0 & = \sqrt{l} \ \widehat{G}_+,\qquad N^1 = -\frac{2 \sqrt{l}}{9}\  \widehat{Q}_+  -\frac{2 \sqrt{l}}{3} :\widehat{H} \widehat{G}_+ + \frac{8 \sqrt{l}}{9} \ \partial  \widehat{G}_+,
\\  M^0 & = \sqrt{l} \ \widehat{G}_- ,\qquad M^1 = \frac{2 \sqrt{l}}{9}\ \widehat{Q}_-  -\frac{2 \sqrt{l}}{3} :\widehat{H}\widehat{G}_- + \frac{\sqrt{l}}{9} \ \partial \widehat{G}_-.
\end{split} \end{equation}

It is straightforward to check that the $l \rightarrow \infty$ limits of these fields are well-defined, and that the limits of $J^k, \Omega^k, N^k, M^k$ coincide with $j^k, w^k, \nu^k, \mu^k$, for $k = 0,1$. By Corollary 6.12 of \cite{CL}, it follows that $\{J^k, \Omega^k, N^k, M^k |\ k = 0,1\}$ strongly generates $C^l(2)$ for generic values of $l$. Finally, it is clear from \eqref{eq:limitfields} that each of the fields $J^k, \Omega^k, N^k, M^k$ for $k = 0,1$, is a normally ordered polynomial in the fields $\widehat{H}, \widehat{L}, \widehat{W}_2, \widehat{W}_3,  \widehat{G}_{\pm}, \widehat{Q}_{\pm}$ and their derivatives, and the same statement holds if we switch $\{J^k, \Omega^k, N^k, M^k |\ k = 0,1\}$ and $\{\widehat{H}, \widehat{L}, \widehat{W}_2, \widehat{W}_3,  \widehat{G}_{\pm}, \widehat{Q}_{\pm}\}$. Therefore, we can replace $\{J^k, \Omega^k, N^k, M^k |\ k = 0,1\}$ with $\{\widehat{H}, \widehat{L}, \widehat{W}_2, \widehat{W}_3,   \widehat{G}_{\pm}, \widehat{Q}_{\pm}\}$ as a set of strong generators of $C^l(2)$ for generic values of $l$. \end{proof}

In the case $n=2$ and $\displaystyle r = -2+ \frac{l+2}{1+l+2}$,  we have $\W^r(\slf_2) \cong Vir^c$ where $\displaystyle c = \frac{l (5 + l)}{(2 + l) (3 + l)}$. Here $Vir^c$ denotes the universal Virasoro vertex algebra with central charge $c$. Similarly, we denote by $Vir_c$ the simple quotient of $Vir^c$. 

\begin{cor} For $l>1$ a positive integer and $(k+1)(l+3)=1$, $\W_k(\slf(3|2))$ is lisse and rational. In particular it is an extension of the rational vertex algebra
$$ V_{\sqrt{6(3+l)} \mathbb{Z}} \otimes Vir_c \otimes \W_s(\slf_l), \qquad c =  \frac{l (5 + l)}{(2 + l) (3 + l)},\qquad s = -l+ \frac{l+2}{l+3}.$$
\end{cor}

\begin{proof} If we replace the complex parameters $k$ and $l$ in $\W^k(\slf(3|2))$ and $C^l(2)$ with formal variables $\bf{k}$ and $\bf{l}$, respectively, Theorem \ref{thm:n=2main} implies that we have an isomorphism of vertex algebras
$$ \W^{\bf{k}}(\slf(3|2)) \rightarrow C^{\bf{l}}(2),$$ which are defined over the (suitably localized) rings $\mathbb{C}[{\bf k}]$ and $\mathbb{C}[\bf{l}]$, respectively. Here $\bf{k}$ and $\bf{l}$ are related by ${(\bf k} +1)({\bf l}+3)=1$.

Recall that for all $k\in \mathbb{C} \setminus \{-1\}$, the specialization of $\W^{\bf{k}}(\slf(3|2))$ at ${\bf k} = k$, coincides with $\W^{k}(\slf(3|2))$. Similarly, by Corollary 6.7 of \cite{CL}, the specialization of $C^{\bf{l}}(2)$ at ${\bf l} = l$ coincides with $C^l(2)$ for all real numbers $l >-2$. (In \cite{CL}, the formal variable $\bf{l}$ was chosen instead so that ${\bf l} = \sqrt{l}$ gave the specialization of $C^{\bf l}(2)$ to $C^l(2)$, but the same proof applies). It follows that for all positive integers $l > 1$, the isomorphism $\W^k(\slf(3|2))\simeq C^l(2)$ of Theorem \ref{thm:n=2main} holds.

Next, by Theorem 8.1 of \cite{CL}, for all positive integers $l > 1$, the natural map $$C^l(2) \rightarrow C_l(2) \cong  \Com(L_{l+1}(\slf_2) \otimes V_{\sqrt{6(3+l)} \mathbb{Z}} , L_l(\slf_{3}) \otimes \E(2)),$$ is surjective. Therefore $\W_k(\slf(3|2)) \cong \Com(L_{l+1}(\slf_2) \otimes V_{\sqrt{6(3+l)} \mathbb{Z}} , L_l(\slf_{3}) \otimes \E(2))$ for all positive integers $l > 1$, and the result follows from Corollary 14.1 of \cite{ACL} in the case $n=2$.  \end{proof}

\begin{rem}  Set $(k+1)(l+3)=1$ and $\displaystyle c = \frac{l (5 + l)}{(2 + l) (3 + l)}$. If the maps \eqref{eq:simplequotientmap} exist for $n=2$, we would have a conformal embedding of simple vertex algebras
$$ \mathcal{H}_0 \otimes Vir_{c} \otimes \W_{s}(\slf_m) \hookrightarrow \W_k(\slf(3|2)),$$
in the following cases.
\begin{enumerate}
\item $m\geq 3$, $\displaystyle l =  -3 + \frac{m - 3}{m-1}$,  and $\displaystyle s = -m + \frac{m-3}{m-1}$, 
\smallskip
\item $m\geq 3$, $\displaystyle l = -3 + \frac{3}{m+1}$, and  $\displaystyle s = -m + \frac{m-2}{m+1}$. 
\end{enumerate}
\end{rem}

\end{document}